\documentclass[unicode,runningheads,envcountsame,a4paper,dvipsnames,12pt]{llncs}

\usepackage{amsmath}
\usepackage{amssymb}
\usepackage{tikz}
\usetikzlibrary{decorations.text, calc}
\usepackage{orcidlink}
\renewcommand{\orcidID}[1]{\orcidlink{#1}}
\usepackage{cleveref}

\spnewtheorem{claim}[theorem]{Claim}{\bfseries}{\itshape}
\Crefname{claim}{Claim}{Claims}

\providecommand{\qedsymbol}{\ensuremath{\square}}
\newenvironment{proofclaim}[1][Proof of the claim]{%
	\par\noindent\textit{#1.}\space\ignorespaces
	\renewcommand{\qedsymbol}{$\lozenge$}%
}{%
	\hfill\qedsymbol\par\medskip
}

\newcommand{\N}{\mathbb{N}}
\newcommand{\Z}{\mathbb{Z}}

\begin{document}
	
	\title{On Edge-Disjoint Maximal Outerplanar Graphs}
	
	\author{
		Yuto Okada\inst{1}\orcidID{0000-0002-1156-0383} \and
		Yota Otachi\inst{1}\orcidID{0000-0002-0087-853X} \and
		Lena Volk\inst{2}\orcidID{0009-0004-3113-9205}
	}
	
	\institute{%
		Nagoya University, Nagoya, Japan.
		\email{pv.20h.3324@s.thers.ac.jp}, 
		\email{otachi@nagoya-u.jp}
		\and
		Technische Universit\"{a}t Darmstadt, Germany.
		\email{volk@mathematik.tu-darmstadt.de}
	}
	
	\maketitle 
	
	\begin{abstract}
		We provide two constructions for \(t\) edge-disjoint maximal outerplanar graphs on every number of \(n \geq 4t\) vertices. The bound on the minimum number of vertices is tight. These constructions yield the existence of optimal outerthick\-ness-\(t\) graphs for every \(t \in \N\).  While one of the constructions works for all values of~\(t\) and extends graphs from Guy and Nowakowski (1990), the other one holds only for powers of~$2$, but yields graphs with maximum degree logarithmic in the number of vertices. Thus, the latter may be helpful in tackling the open question of determining the outerthickness of all complete graphs.
	\end{abstract}
	
	\section{Introduction}
	
	In their study of the thickness of complete graphs, Beineke and Harary~\cite{Beineke_Harary_1965} presented a construction of \(t\) edge-disjoint maximal planar graphs on  \(6t\) vertices. This was extended by Boswell and Simpson~\cite{DBLP:journals/dm/BoswellS98} to a construction of \(t\) edge-disjoint maximal planar graphs on every number of \(n \geq 6t\) vertices. Crucially, such a construction is not possible on less than \(6t-1\) vertices~\cite{DBLP:journals/dm/BoswellS98}, making the result almost tight. In this paper, we consider the analogous question for outerplanar graphs.
	
	We present two constructions of edge-disjoint maximal outerplanar graphs.
	For the first one, we extend a construction of Guy and Nowakowski~\cite{Guy1990} to obtain~\(t\) edge-disjoint maximal outerplanar graphs on \(4t\) vertices, 
	and then we show that every set of~\(t\) edge-disjoint maximal outerplanar graphs on \(n\) vertices can be extended to such a set on \(n+1\) vertices.
	These together imply that for every~$t \ge 1$ and every $n \ge 4t$, there exist $t$ edge-disjoint maximal outerplanar graphs on $n$ vertices.
	The lower bound of \(4t\) on the number of vertices is tight when $t \ge 2$.
	For the second construction, we restrict $t$ to be a power of~$2$ 
	and construct another set of \(t\) edge-disjoint maximal outerplanar graphs on \(4t\) vertices from scratch.
	The second set differs from the first one by having a logarithmic maximum degree instead of a linear one in each of the outerplanar graphs.
	
	Besides its theoretical interest, the second construction is also motivated by the open problem of determining the outerthickness of complete graphs.
	This is closely related to our constructions, as they both yield decompositions of the complete graph~\(K_{4t}\) minus a matching (of size~\(t\)) into~\(t\) outerplanar graphs.
	Guy and Nowakowski~\cite{Guy1990} determined the outerthickness of $K_{n}$ in most cases,
	but as they stated in the follow-up paper~\cite{GuyNowakowski1990b}, the case of~\(n \equiv 3 \bmod 4\) actually remains open.
	It appears that their decompositions, which our first construction heavily relies on, cannot be easily extended to the case when~${n \equiv 3 \bmod 4}$.
	Hence, due to its significantly lower maximum degree and the overall different structure, our second construction may be helpful to determine the outerthickness of some of the unsolved cases of complete graphs. Furthermore, there are several interesting consequence from these two constructions.
	
	The study of optimal graphs of a given type, i.e., graphs with the maximum possible number of edges, is a classical and widely investigated problem in graph drawing. From our constructions, we immediately obtain for every~\(t \in \N\) the existence of outerthick\-ness-\(t\) graphs with \(t(2n-3)\) edges on every number of~\(n \geq 4t\) vertices. This number of edges is the maximum possible, as every outerplanar graph has at most \(2n-3\) edges, making these graphs optimal outerthickness-\(t\) graphs. While every maximal outerplanar graph (one where no edge can be added without violating outerplanarity) is also optimal~\cite[Corollary~11.9(a)]{Harary1969book}, we give examples in this paper that for \(t\geq2\) the notions of maximal and optimal outerthickness-\(t\) graphs do not coincide.
	
	Since all planar graphs edge-decompose into two outerplanar graphs~\cite{DBLP:conf/stoc/Goncalves05}, the graphs with outerthickness-2 are of particular interest. From our constructions, the existence of an infinite family of outerthickness-2 graphs which are not 1-planar follows. This separates these two generalizations of planar graphs. For the converse, every 1-planar graph admits a partition into a planar graph and a forest~\cite{DBLP:journals/dam/Ackerman14}. Hence, every 1-planar graph has outerthickness at most~3, but it is still open whether there are 1-planar graphs with outerthickness precisely 3. 
	Gethner and Sulanke~\cite{DBLP:journals/gc/GethnerS09} showed that the maximum chromatic number of outerthickness-2 graphs is in \(\{6,7,8\}\) and raised the question of increasing the lower bound. Since all 1-planar graphs are 6-colorable~\cite{borodin1984solution}, graphs with outerthickness-2 which are not 1-planar are of special interest towards this question.
	
	\section{Preliminaries}
	
	We assume that the reader is familiar with the standard terminology of graph theory (for instance, see \cite{Diestel2025}).
	For a positive integer \(n\), let \([n]_0\) denote the set of nonnegative integers smaller than \(n\), i.e., \([n]_0 = \{0,\dots, n-1\}\).
	
	\paragraph{Outerplanar graphs.}
	
	A graph $G$ is an \emph{outerplanar} graph if it admits a drawing in the plane such that there is no edge crossing and all vertices appear on the outer face.
	An outerplanar graph is \emph{maximal} if we cannot add an edge without breaking its outerplanarity.
	It is known that every maximal outerplanar graph on $n \geq 2$ vertices has exactly $2n-3$ edges~\cite[Corollary~11.9(a)]{Harary1969book}.
	A maximal outerplanar graph with at least three vertices has a unique Hamiltonian cycle on the outer face in every drawing in the plane~\cite{SYSLO197947}. We call this cycle the \emph{outer cycle}.
	
	\paragraph{Edge-disjoint outerplanar graphs.} 
	By \emph{\(t\) edge-disjoint outerplanar graphs on \(n\) vertices}, we refer to \(t\) outerplanar graphs on the same vertex set \(V\) with~\(|V|=n\) such that all the edge sets of the graphs are pairwise disjoint.
	
	\paragraph{Outerthickness.}
	
	The \emph{outerthickness} of a graph $G$ is the minimum number $t$ such that the edge set $E(G)$ admits a partition $\{E_0, E_1, \dots, E_{t-1}\}$ where the graph $(V(G), E_i)$ is an outerplanar graph for every $i \in [t]_0$.
	Note that in this case~$|E(G)| \leq t (2 |V(G)| - 3)$ holds by definition when $|V(G)| \geq 2$.
	
	Let $G$ be a graph with $n \geq 3$ vertices, $m$ edges, and outerthickness $t \geq 1$.
	We call~$G$ a \emph{maximal outerthickness-$t$} graph if no edge can be added to $G$ without increasing its outerthickness.
	We call~$G$ an \emph{optimal outerthickness-$t$} graph if $m = t (2n-3)$.
	Note that every optimal outerthickness-\(t\) graph has an edge partition into \(t\) maximal outerplanar graphs. Throughout the paper, we consider the notions of maximality and optimality regarding outerthickness only for graphs with at least three vertices.

	\paragraph{1-planar graphs.} A graph is called \textit{1-planar} if it can be drawn in the plane such that each edge is crossed by at most one other edge.

	\section{Edge-disjoint maximal outerplanar graphs}
	\label{sec:edge-dis-max-out-plan}
	
	We begin by giving two different constructions for \(t\) edge-disjoint maximal outerplanar graphs on \(4t\) vertices. The first construction holds for all \(t \geq 1\) and extends the graphs in the proof of Theorem 1 in~\cite{Guy1990}. Since the maximum degree of the outerplanar graphs in this construction is linear in the number of vertices, we provide a second construction where the maximum degree is only logarithmic in the number of vertices. This construction is only possible whenever~\(t\) is a power of~$2$.
		
	\begin{lemma}
		\label{lem:optimal-ot-construct}
		For every $t \geq 1$, there exist \(t\) edge-disjoint maximal outerplanar graphs on~$4t$ vertices with maximum degree \(t+3\).
	\end{lemma}
	
	\begin{proof}
		We consider graphs from the proof of Theorem 1 in~\cite{Guy1990}. For the ease of transition to their construction let \(r=t\). For each \(i \in [r]_0\), Guy and Nowakowski~\cite{Guy1990} construct an outerplanar graph, which they call \emph{graph $i$}, on~\(4r\) vertices and with \(2 \cdot 4r - 4\) edges. These graphs are depicted in \Cref{fig:graphsfromGuy} and shown to be pairwise edge-disjoint in \cite{Guy1990}. Note that these graphs have maximum degree \(r+2\), which is attained at the vertices \(i\), \(i+r\), \(i+2r\) and \(i+3r\). Now, for each~\(i \in [r]_0\), we add the edge~\(\{i,i+2r\}\) to graph~\(i\). This preserves outerplanarity of graph \(i\), since the edge can be embedded as the diagonal of the square in \Cref{fig:graphsfromGuy}. (In~\cite{Guy1990}, this edge is called dexter diameter of graph \(i\).) Overall, this yields~\(r\) many maximal outerplanar graphs. Further, these graphs remain edge-disjoint as none of the graphs contained edges with difference \(2r\) between the two vertices before~\cite{Guy1990}. The maximum degree of these graphs including the diagonals is~\(r+3\) and is attained at the vertices~\(i\) and \(i+2r\).
	\end{proof}
	
	\begin{figure}[bt]
		\centering
		\begin{tikzpicture}[scale=1.6]
			
			\def\a{1.25}        
			\def\n{9}        
			
			\coordinate (O)  at (-\a, \a);
			\coordinate (R)  at ( \a, \a);
			\coordinate (2R) at ( \a,-\a);
			\coordinate (3R) at (-\a,-\a);
			
			\node[circle, fill=black, inner sep=1.5pt] at (-\a, \a) {};
			\node[below right]  at (O)  {$0$};
			\node[circle, fill=black, inner sep=1.5pt] at (R) {};
			\node[below left] at (R)  {$r$};
			\node[circle, fill=black, inner sep=1.5pt] at (2R) {};
			\node[above left] at (2R) {$2r$};
			\node[circle, fill=black, inner sep=1.5pt] at (3R) {};
			\node[above right]  at (3R) {$3r$};
			
			\draw[thick] (O) -- (R) -- (2R) -- (3R) -- cycle;
			
			\draw[thick] (O) arc[start angle=180, end angle=0, radius=\a];
			\draw[thick] (R) arc[start angle=90, end angle=-90, radius=\a];
			\draw[thick] (2R) arc[start angle=0, end angle=-180, radius=\a];
			\draw[thick] (3R) arc[start angle=-90, end angle=-270, radius=\a];
			
			\foreach \k in {2,3,6,7,...,\n}{
				
				\coordinate (A\k) at ({\a*cos(180-180*\k/(\n+1))},{\a+\a*sin(180-180*\k/(\n+1))}); 
				\node[circle, fill=black, inner sep=1.5pt] at (A\k) {};
				\path (O) -- (A\k) coordinate[pos=1.2] (LA\k);
				
				\coordinate (B\k) at ({\a+\a*cos(90-180*\k/(\n+1))},
				{\a*sin(90-180*\k/(\n+1))});
				\node[circle, fill=black, inner sep=1.5pt] at (B\k) {};	
				\path (R) -- (B\k) coordinate[pos=1] (LB\k);
				
				\coordinate (C\k) at ({\a*cos(-180*\k/(\n+1))},{-\a+\a*sin(-180*\k/(\n+1))});
				\node[circle, fill=black, inner sep=1.5pt] at (C\k) {};	
				\path (2R) -- (C\k) coordinate[pos=1.2] (LC\k);
				
				\coordinate (D\k) at ({-\a+\a*cos(-90-180*\k/(\n+1))}, {\a*sin(-90-180*\k/(\n+1))});
				\node[circle, fill=black, inner sep=1.5pt] at (D\k) {};	
				\path (3R) -- (D\k) coordinate[pos=1] (LD\k);
				
			}
			
			\foreach \k in {3,6,7,...,\n}{
				\draw (O) -- (A\k);
				\draw (R) -- (B\k);	
				\draw (2R) -- (C\k);
				\draw (3R) -- (D\k);
			}
			
			\node[above left] at (A2) {$\lfloor 3r/2 \rfloor$};
			\node[] at (LA3) {};
			\node[] at (LA6) {$r+2$};
			\node[] at (LA7) {$2r-2$};
			\node[] at (LA8) {$r+1$};
			\node[] at (LA9) {$2r-1$};
			
			\node[above right] at (LB2) {$\lfloor 5r/2 \rfloor$};
			\node[] at (LB3) {};
			\node[below right] at (LB6) {$2r+2$};
			\node[below right] at (LB7) {$3r-2$};
			\node[below right] at (LB8) {$2r+1$};
			\node[below right] at (LB9) {$3r-1$};
			
			\node[right] at (LC2) {$\lfloor 7r/2 \rfloor$};
			\node[] at (LC3) {};
			\node[] at (LC6) {$3r+2$};
			\node[] at (LC7) {$4r-2$};
			\node[] at (LC8) {$3r+1$};
			\node[] at (LC9) {$4r-1$};
			
			\node[below left] at (D2) {$\lfloor r/2 \rfloor$};
			\node[] at (LD3) {};
			\node[above left] at (LD6) {$2$};
			\node[above left] at (LD7) {$r-2$};
			\node[above left] at (LD8) {$1$};
			\node[above left] at (LD9) {$r-1$};
			
			\node at (-0.5,2.1) {\(\dots\)};
			\node at (0.5,-2.1) {\(\dots\)};
			\node at (2.1,0.5) {\(\vdots\)};
			\node at (-2.1,-0.5) {\(\vdots\)};
		\end{tikzpicture}
		
		\caption{The so-called \emph{graph zero} from \cite{Guy1990}. For \(i \in [r]_0\), the so-called \emph{graph \(i\)} is a copy of graph zero obtained by increasing the label of each vertex by \(i \bmod 4r\). These graphs are shown to be pairwise edge-disjoint in \cite{Guy1990} and have maximum degree \(r+2\).}
		\label{fig:graphsfromGuy}
	\end{figure}
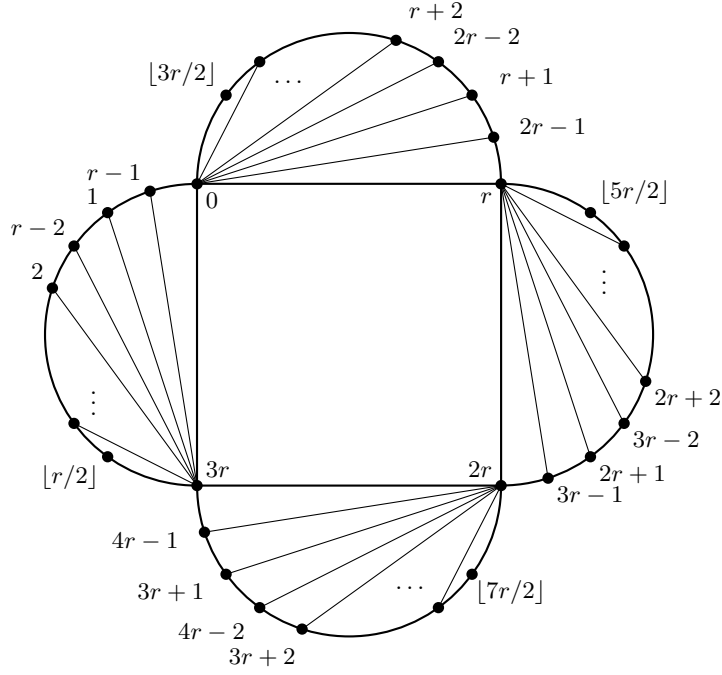
	
	\begin{figure}[t]
		\centering
		\begin{tikzpicture}[scale=0.8]
			\def \n {4} 
			\def \h {3} 
			\def \radius {1.5cm} 
			\def\labeldist{0.35cm}  
			
			\foreach \i in {0,...,\h} {
				\node[draw, circle, fill=black, inner sep=1.5pt] (\i) at ({360/\n * \i}: \radius) {};
				
				\node[font=\small]
				at ({360/\n * \i}:{\radius + \labeldist}) {\i};
			}
			
			\foreach \i/\j in {0/1, 1/2, 2/3, 3/0} {
				\draw (\i) -- (\j);
			}
			
			\foreach \i/\j in {0/2} {
				\draw (\i) -- (\j);
			}
		\end{tikzpicture}
		\hspace{0.2cm}
		\begin{tikzpicture}[scale=0.8]
			\def \n {8} 
			\def \h {7} 
			\def \radius {1.5cm} 
			\def\labeldist{0.35cm}  
			
			\foreach \i in {0,2,4,6} {
				\node[draw, circle, fill=black, inner sep=1.5pt] (\i) at ({360/\n * \i}: \radius) {};
				
				\node[font=\small]
				at ({360/\n * \i}:{\radius + \labeldist}) {\i};
			}
			
			\foreach \i in {1,3,5,7} {
				\node[draw, blue, circle, fill=blue, inner sep=1.5pt] (\i) at ({360/\n * \i}: \radius) {};
				
				\node[font=\small, text=blue]
				at ({360/\n * \i}:{\radius + \labeldist}) {\i};
			}
			
			\foreach \i/\j in {0/1, 1/2, 2/3, 3/4, 4/5, 5/6, 6/7, 7/0} {
				\draw[blue] (\i) -- (\j);
			}
			
			\foreach \i/\j in {0/2, 2/4, 4/6, 6/0, 0/4} {
				\draw (\i) -- (\j);
			}
		\end{tikzpicture}
		\hspace{0.2cm}
		\begin{tikzpicture}[scale=0.8]
			\def \n {8} 
			\def \h {7} 
			\def \radius {1.5cm} 
			\def\labeldist{0.35cm}  
			
			\foreach \i in {0,2,4,6} {
				\node[draw, circle, fill=black, inner sep=1.5pt] (\i) at ({360/\n * \i}: \radius) {};
				
				\pgfmathtruncatemacro{\label}{mod(5*\i +1,8)}
				
				\node[font=\small]
				at ({360/\n * \i}:{\radius + \labeldist}) {\label};
			}
			
			\node[draw, blue, circle, fill=blue, inner sep=1.5pt] (1) at ({360/\n * 1}: \radius) {};
			\node[font=\small, text=blue] at ({360/\n * 1}:{\radius + \labeldist}) {6};
			\node[draw, blue, circle, fill=blue, inner sep=1.5pt] (3) at ({360/\n * 3}: \radius) {};
			\node[font=\small, text=blue] at ({360/\n * 3}:{\radius + \labeldist}) {0};
			\node[draw, blue, circle, fill=blue, inner sep=1.5pt] (5) at ({360/\n * 5}: \radius) {};
			\node[font=\small, text=blue] at ({360/\n * 5}:{\radius + \labeldist}) {2};
			\node[draw, blue, circle, fill=blue, inner sep=1.5pt] (7) at ({360/\n * 7}: \radius) {};
			\node[font=\small, text=blue] at ({360/\n * 7}:{\radius + \labeldist}) {4};
			
			\foreach \i/\j in {0/1, 1/2, 2/3, 3/4, 4/5, 5/6, 6/7, 7/0} {
				\draw[blue] (\i) -- (\j);
			}
			
			\foreach \i/\j in {0/2, 2/4, 4/6, 6/0, 0/4} {
				\draw (\i) -- (\j);
			}
		\end{tikzpicture}
		\caption{On the left the base graph \((V^0,E^0_0)\) from the proof of~\Cref{lem:optimal-ot-power-2-construct} is depicted. The graph~\((V^1,E^1_0)\) in the middle, and \((V^1,E^1_1)\) on the right, are obtained from \((V^0,E^0_0)\) as described in the proof of~\Cref{lem:optimal-ot-power-2-construct}.}
		\label{fig:labeling-8}
	\end{figure}
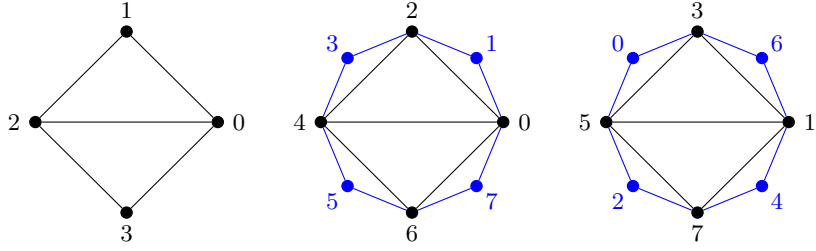
	
	\begin{lemma}
		\label{lem:optimal-ot-power-2-construct}
		For every $s \geq 0$, there exist \(2^s\) edge-disjoint maximal outerplanar graphs on~$4 \cdot 2^s$ vertices with maximum degree \(2s+3\).
	\end{lemma}
	\begin{proof}
		Let \(s \in \N_0\) and let \(V^s=[2^{s+2}]_0\). First we give~\(2^s\) maximal outerplanar graphs on the vertex set \(V^s\). Namely, for each \(k \in [2 ^s]_0\), we construct a graph~\((V^s,E^s_k)\). Then we show that the edge sets \(E^s_0, \dots, E^{s}_{2^s-1}\) are pairwise disjoint.		
		We say an outerplanar graph with vertex set \([n]_0\) has property \((\star)\) if for some~\(d \in [n]_0\) the outer cycle is of the form~\((0,d,2d,\dots,(n-1)d,0)\) (all entries considered as values in \(\Z_n\)). Note that in this case \(d\) has to be an element of order \(n\) in \(\Z_n\), i.e., if \(n=2^{s+2}\) then \(d\) is odd.
		Below, we construct the graphs $(V^s,E^s_k)$ so that each of them has property \((\star)\) for some $d$.
		
		We define the graphs \((V^s,E^s_k)\) for \(k \in [2^s]_0\) by induction on \(s\) (see \Cref{fig:labeling-8,fig:labeling-induction-step,fig:labeling-16}). For~\(s=0\), the single graph \((V^0,E^0_0)\) is depicted in~\Cref{fig:labeling-8}. This graph has property~\((\star)\) for~\(d=1\). 
		Now, let $s \ge 0$. From the current set of graphs $\{(V^{s}, E_{k}^{s}) \mid k \in [2^{s}]_{0}\}$ with property $(\star)$,
		we obtain the next set of graphs $\{(V^{s+1}, E_{k}^{s+1}) \mid k \in [2^{s+1}]_{0}\}$ with property~$(\star)$
		by constructing two new graphs $(V^{s+1}, E_{k}^{s+1})$ and $(V^{s+1}, E_{2^{s}+k}^{s+1})$ from $(V^{s}, E_{k}^{s})$ for each~$k \in [2^{s}]_{0}$ as follows.
		\begin{itemize}
			\item Start with the graph \((V^s,E_k^s)\). Assume that $(V^{s}, E_{k}^{s})$ has property $(\star)$ for some odd~$d = d_{k}^{s} \in [2^{s+2}]_{0}$.
			\item For $(V^{s+1}, E_{k}^{s+1})$, double the label of all vertices. For $(V^{s+1}, E_{2^{s}+k}^{s+1})$, double the label of all vertices and add~$1$. Note that all these new labels belong to \([2^{s+3}]_0\).
			\item The outer cycle of the new graphs is of the form $(0, (2d_k^s), 2(2d_k^s), \dots, (n-1)(2d_k^s), 0)$, where all entries are considered as values in \(\Z_{2^{s+3}}\).
			The equation \(2x \equiv 2d_k^s \bmod 2^{s+3}\) has two solutions in \([2^{s+3}]_0\), namely \(x=d_k^s\) and \(y=d_k^s+2^{s+2}\). Since \(d_k^s\) is odd, both solutions are odd, and hence generators of \(\Z_{2^{s+3}}\).
			\item For each edge \(\{u,u+2d_k^s\}\) (modulo \(2^{s+3}\)) on the outer cycle we add for $(V^{s+1}, E_{k}^{s+1})$ the vertex \(u+x \bmod 2^{s+3}\) and for $(V^{s+1}, E_{2^{s}+k}^{s+1})$ the vertex \(u+y \bmod 2^{s+3}\) to the graph and connect it to the two endpoints of the edge. See~\Cref{fig:labeling-induction-step} for an illustration of the labeling.
			\item Observe that both $(V^{s+1}, E_{k}^{s+1})$ and $(V^{s+1}, E_{2^{s}+k}^{s+1})$ have property \((\star)\) for $d = x$ and~$d = y$, respectively. 
			Further, both graphs now contain all vertices \([2^{s+3}]_0\) as \(x\) and~\(y\) are generators of~\(\Z_{2^{s+3}}\).
		\end{itemize}
		Overall we constructed a maximal outerplanar graph \((V^{s+1},E_k^{s+1})\)  for each \({k \in [2^{s+1}]_0}\). The inductive construction from 0 to 1 is visualized in~\Cref{fig:labeling-8}, and from 1 to 2 in~\Cref{fig:labeling-16}. Before we proceed to show that the edge sets \(E^s_0, \dots, E^{s}_{2^s-1}\) are pairwise disjoint, we need two auxiliary statements about the graphs \((V^s,E_k^s)\).
		
		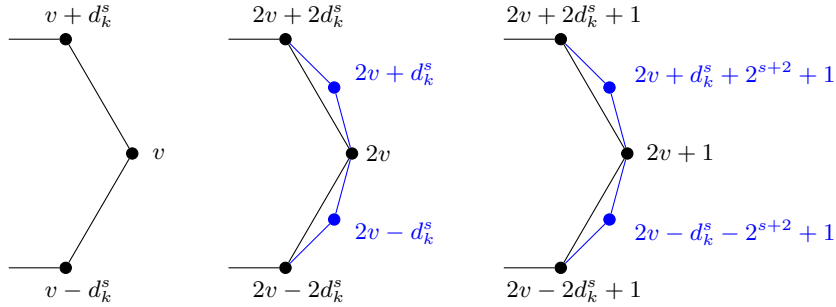
\begin{figure}[t]
			\centering
			\begin{tikzpicture}[scale=1]
				\def \n {6} 
				\def \h {5} 
				\def \radius {1.75cm} 
				\def\labeldist{0.35cm}  
				
				\foreach \i in {0,1,5} {
					\node[draw, circle, fill=black, inner sep=1.5pt] (\i) at ({360/\n * \i}: \radius) {};
				}
				
				\foreach \i/\j in {0/1, 5/0} {
					\draw (\i) -- (\j);
				}
				
				\draw (1) -- ++(-0.75,0);
				\draw (5) -- ++(-0.75,0);
				
				\node[font=\small] at ({360/\n * 0}:{\radius + \labeldist}) {\(v\)};
				\node[font=\small] at ({360/\n * 1}:{\radius + \labeldist}) {\(v+d_k^s\)};
				\node[font=\small] at ({360/\n * 5}:{\radius + \labeldist}) {\(v-d_k^s\)};
			\end{tikzpicture}
			\hspace{0.5cm}
			\begin{tikzpicture}[scale=1]
				\def \n {12} 
				\def \h {11} 
				\def \radius {1.75cm} 
				\def\labeldist{0.35cm}  
				
				\foreach \i in {0,2,10} {
					\node[draw, circle, fill=black, inner sep=1.5pt] (\i) at ({360/\n * \i}: \radius) {};
				}
				
				\foreach \i in {1,11} {
					\node[draw, blue, circle, fill=blue, inner sep=1.5pt] (\i) at ({360/\n * \i}: \radius) {};
				}
				
				\foreach \i/\j in {0/1, 1/2, 10/11, 11/0} {
					\draw[blue] (\i) -- (\j);
				}
				
				\foreach \i/\j in {0/2, 10/0} {
					\draw (\i) -- (\j);
				}
				
				\draw (2) -- ++(-0.75,0);
				\draw (10) -- ++(-0.75,0);
				
				\node[font=\small] at ({360/\n * 0}:{\radius + \labeldist}) {\(2v\)};
				\node[font=\small] at ({360/\n * 2}:{\radius + \labeldist}) {\(2v+2d_k^s\)};
				\node[font=\small] at ({360/\n * 10}:{\radius + \labeldist}) {\(2v-2d_k^s\)};
				\node[font=\small, xshift=0.5cm, text=blue] at ({360/\n * 1}:{\radius + \labeldist}) {\(2v+d_k^s\)};
				\node[font=\small, xshift=0.5cm, text=blue] at ({360/\n * 11}:{\radius + \labeldist}) {\(2v-d_k^s\)};
			\end{tikzpicture}
			\hspace{0.5cm}
			\begin{tikzpicture}[scale=1]
				\def \n {12} 
				\def \h {11} 
				\def \radius {1.75cm} 
				\def\labeldist{0.35cm}  
				
				\foreach \i in {0,2,10} {
					\node[draw, circle, fill=black, inner sep=1.5pt] (\i) at ({360/\n * \i}: \radius) {};
				}
				
				\foreach \i in {1,11} {
					\node[draw, blue, circle, fill=blue, inner sep=1.5pt] (\i) at ({360/\n * \i}: \radius) {};
				}
				
				\foreach \i/\j in {0/1, 1/2, 10/11, 11/0} {
					\draw[blue] (\i) -- (\j);
				}
				
				\foreach \i/\j in {0/2, 10/0} {
					\draw (\i) -- (\j);
				}
				
				\draw (2) -- ++(-0.75,0);
				\draw (10) -- ++(-0.75,0);
				
				\node[font=\small, xshift=0.35cm] at ({360/\n * 0}:{\radius + \labeldist}) {\(2v+1\)};
				\node[font=\small] at ({360/\n * 2}:{\radius + \labeldist}) {\(2v+2d_k^s+1\)};
				\node[font=\small] at ({360/\n * 10}:{\radius + \labeldist}) {\(2v-2d_k^s+1\)};
				\node[font=\small, xshift=1.35cm, text=blue] at ({360/\n * 1}:{\radius + \labeldist}) {\(2v+d_k^s+2^{s+2}+1\)};
				\node[font=\small, xshift=1.35cm, text=blue] at ({360/\n * 11}:{\radius + \labeldist}) {\(2v-d_k^s-2^{s+2}+1\)};
			\end{tikzpicture}
			\caption{The inductive vertex labeling in the proof of~\Cref{lem:optimal-ot-power-2-construct}. On the left, the part of the outer cycle in \((V^s,E_k^s)\) next to \(v \in V^s\) is depicted.  In the middle the corresponding part of the inductively defined graph \((V^{s+1},E_k^{s+1})\) is depicted. For the graph in the middle the labels of the black vertices get doubled (modulo \(2^{s+3}\)) and the new vertices and edges in blue are added (for \(u=2v\) in the fourth step of the construction). Analogously, on the right, the construction of \((V^{s+1},E_{2^s+k}^{s+1})\) is depicted.}
			\label{fig:labeling-induction-step}
		\end{figure}
		
		\begin{claim}
			For every odd \(d \in [2^{s+2}]_0\), there is some \(k \in [2^s]_0\) such that the outer cycle in~\((V^s,E_k^s)\) is \((0,d,2d,\dots,(2^{s+2}-1)d,0)\) (all entries considered as values in \(\Z_{2^{s+2}}\)).
			\label{claim:odd}
		\end{claim}
		
		\begin{proofclaim}
			We show this by induction on \(s\). The base case \(s=0\) follows from \Cref{fig:labeling-8} as the outer cycle \((0,1,2,3,0)\) of the graph~\((V^0,E_0^0)\) realizes $d = 1$ (from left to right) and $d = 3$ (from right to left). Assume that the statement is true for some \(s \in \N_0\). Let~\(d \in [2^{s+3}]_0\) odd. If \(d \in [2^{s+2}]_0\), by induction hypothesis there is some \(k \in [2^s]_0\) such that the outer cycle in \((V^s,E_k^s)\) is~\((0,d,2d,\dots,(2^{s+2}-1)d,0)\) (all entries considered as values in \(\Z_{2^{s+2}}\)). Then, by the fourth step of our construction, the outer cycle in \((V^{s+1},E_k^{s+1})\) is~\((0,d,2d,\dots,(2^{s+3}-1)d,0)\) (all entries considered as values in \(\Z_{2^{s+3}}\)) as well. If \(d \geq 2^{s+2}\), we write \(d=2^{s+2}+d'\) for some~\(d' \in [2^{s+2}]_0\). Again, by induction hypothesis there is some~\(k' \in [2^s]_0\) such that the the outer cycle in \((V^s,E_{k'}^s)\) is \((0,d',2d',\dots,(2^{s+2}-1)d',0)\) (all entries considered as values in \(\Z_{2^{s+2}}\)). Then for \(k=2^s+k' \in [2^{s+1}]_0\), again by the fourth step of the construction, the outer cycle in \((V^{s+1},E_k^{s+1})\) is \((0,d,2d,\dots,(2^{s+3}-1)d,0)\) (all entries considered as values in \(\Z_{2^{s+3}}\)).
		\end{proofclaim}
		
		\begin{claim}
			For every \(i \in [2^s]_0\), there exists some \(k \in [2^s]_0 \) such that \(\{i,i+2^{s+1}\} \in E^{s}_k\).
			\label{claim:middle}
		\end{claim}
		
		\begin{proofclaim}
			We show this by induction on \(s\) again. The base case \(s=0\) follows from \Cref{fig:labeling-8} as the edge \(\{0,2\} \in E_0^0\). Assume the claim holds for some~\(s \in \N_0\). Now let~\(i \in [2^{s+1}]_0\). If \(i\) is even, by induction hypothesis there is some \(k \in [2^s]_0 \) such that the edge \(\{i/2,i/2+2^{s+1}\} \in E^{s}_k\). Then by construction of~\(E^{s+1}_k\), the edge \(\{i,i+2^{s+2}\} \in E^{s+1}_k\).
			If \(i\) is odd, we do the same to obtain some \(k \in [2^s]_0 \) such that the edge \(\{(i-1)/2,(i-1)/2+2^{s+1}\} \in E^{s}_k\). Then by construction of \(E^{s+1}_{2^s+k}\), the edge \(\{i,i+2^{s+2}\} \in E^{s+1}_{2^s+k}\).
		\end{proofclaim}
		
		As a next step, we show that the edge sets \(E^s_0, \dots, E^{s}_{2^s-1}\) are pairwise disjoint. To prove this, we define a graph \(G^s\) on the vertex set~\(V^s\) and with \(2^s(2n-3)\) edges and then show that each edge of \(G^s\) is in one of the edge sets \(E^s_0, \dots, E^{s}_{2^s-1}\). Note that by counting argument this suffices to show that the edge sets are pairwise disjoint. 
		
		Let $G^s$ be the graph with vertex set~\(V^s\) such that two vertices $u, v \in V^{s} = [2^{s+2}]_0$ are adjacent if and only if $u-v \not\equiv 2^{s+1} \bmod 2^{s+2}$.
		Further, we add for each $i \in [2^{s}]_0$, the edge~$\{i,i+2^{s+1}\}$ to $G^s$. Then, we have
		\begin{align*}
			|E(G^s)|=\left(2^{s+1}-1\right)2^{s+2}+2^{s}=2^{s}(2^{s+3}-3).
		\end{align*} Next, we show that each edge of $G^s$ is contained in some \(E^s_k\) for \(k \in [2^s]_0\). Again, we show this by induction on \(s\). 
		
		For the base case of \(s=0\), the statement can be easily derived from the graph~\((V^0,E^0_0)\) depicted in \Cref{fig:labeling-8}. Now assume the statement holds true for some \(s \in \N_0\). Let~\(\{u,v\} \in E(G^{s+1})\). We aim to show that there is some \(k \in [2^{s+1}]_0\) such that \(\{u,v\} \in E^{s+1}_k.\) First, assume that \(u-v \equiv 2^{s+2} \bmod 2^{s+3}\). Then by definition of \(G^{s+1}\), we have \(\{u,v\}=\{i,i+2^{s+2}\}\) for some \(i \in [2^{s+1}]_0\). Then by~\Cref{claim:middle} there is some \(k \in [2^{s+1}]_0\) such that~\(\{u,v\} \in E^{s+1}_k.\) Hence in the following, we assume that \(u-v \not \equiv 2^{s+2} \bmod 2^{s+3}\).
		\begin{itemize}
			\item If \(u\) and \(v\) are both even, then \(\{u/2,v/2\} \in E(G^{s})\) as \(u/2-v/2 \not \equiv 2^{s+1} \bmod 2^{s+2}\). Hence, by induction hypothesis, there exists some \(k \in [2^s]_0\) such that \(\{u/2,v/2\} \in E_k^s\). Then we have \(\{u,v\} \in E_k^{s+1}\) by construction of~\(E_k^{s+1}\).
			\item If \(u\) and \(v\) are both odd, then \(\{(u-1)/2,(v-1)/2\} \in E(G^{s})\) as we have \((u-1)/2-(v-1)/2 \not \equiv 2^{s+1} \bmod 2^{s+2}\). Again, by induction hypothesis, there exists some \(k \in [2^s]_0\) such that \(\{(u-1)/2,(v-1)/2\} \in E_k^s\). Then we have \(\{u,v\} \in E_{2^s+k}^{s+1}\) by construction of \(E_{2^s+k}^{s+1}\).
			\item If one of \(u\) and \(v\) is odd and the other one is even, then \(d \equiv u-v \bmod 2^{s+3}\) is odd. By~\Cref{claim:odd} there is some \(k \in [2^{s+1}]_0\) such that the edge \(\{u,v\}\) is on the outer cycle of~\(E^{s+1}_k\).
		\end{itemize}
		Hence, every edge of \(G^{s+1}\) is contained in some edge set \(E^{s+1}_k\) for \(k \in [2^{s+1}]_0\). This completes the induction. Thus, the edge sets \(E^s_0, \dots, E^{s}_{2^s-1}\) are pairwise disjoint. 
		
		Finally, let us take a look at the maximum degree of the graphs~\((V^s,E^s_k)\). First, see~\Cref{fig:labeling-8} to note that the graph \((V^0,E_0^0)\) has maximum degree~3. Further, with each induction step, the maximum degree of the graph is increased by~$2$. More precisely, when constructing the two new graphs $(V^{s+1}, E_{k}^{s+1})$ and~$(V^{s+1}, E_{2^{s}+k}^{s+1})$ from $(V^{s}, E_{k}^{s})$, only in the fourth step the degree of vertices is changed at all, and in this step it is increased by~$2$. Thus, the maximum degree of~\((V^s,E^s_k)\) is \(2s+3\) for all \(k \in [2^s]_0\).
	\end{proof}
	
	\begin{figure}[t]
		\centering
		\begin{tikzpicture}[scale=0.65]
			\def \n {16} 
			\def \h {15} 
			\def \radius {1.5cm} 
			\def \newradius {1.75cm} 
			\def\labeldist{0.5cm}  
			
			\node[draw, circle, fill=black, inner sep=1.5pt] (0) at ({360/\n * 0}: \radius) {};
			\node[font=\small] at ({360/\n * 0}:{\radius + \labeldist}) {0};
			
			\node[draw, blue, circle, fill=blue, inner sep=1.5pt] (1) at ({360/\n * 1}: \newradius) {};
			\node[font=\small, text=blue] at ({360/\n * 1}:{\newradius + \labeldist}) {1};
			
			\node[draw, circle, fill=black, inner sep=1.5pt] (2) at ({360/\n * 2}: \radius) {};
			\node[font=\small] at ({360/\n * 2}:{\radius + \labeldist}) {2};
			
			\node[draw, blue, circle, fill=blue, inner sep=1.5pt] (3) at ({360/\n * 3}: \newradius) {};
			\node[font=\small, text=blue] at ({360/\n * 3}:{\newradius + \labeldist}) {3};
			
			\node[draw, circle, fill=black, inner sep=1.5pt] (4) at ({360/\n * 4}: \radius) {};
			\node[font=\small] at ({360/\n * 4}:{\radius + \labeldist}) {4};
			
			\node[draw, blue, circle, fill=blue, inner sep=1.5pt] (5) at ({360/\n * 5}: \newradius) {};
			\node[font=\small, text=blue] at ({360/\n * 5}:{\newradius + \labeldist}) {5};
			
			\node[draw, circle, fill=black, inner sep=1.5pt] (6) at ({360/\n * 6}: \radius) {};
			\node[font=\small] at ({360/\n * 6}:{\radius + \labeldist}) {6};
			
			\node[draw, blue, circle, fill=blue, inner sep=1.5pt] (7) at ({360/\n * 7}: \newradius) {};
			\node[font=\small, text=blue] at ({360/\n * 7}:{\newradius + \labeldist}) {7};
			
			\node[draw, circle, fill=black, inner sep=1.5pt] (8) at ({360/\n * 8}: \radius) {};
			\node[font=\small] at ({360/\n * 8}:{\radius + \labeldist}) {8};
			
			\node[draw, blue, circle, fill=blue, inner sep=1.5pt] (9) at ({360/\n * 9}: \newradius) {};
			\node[font=\small, text=blue] at ({360/\n * 9}:{\newradius + \labeldist}) {9};
			
			\node[draw, circle, fill=black, inner sep=1.5pt] (10) at ({360/\n * 10}: \radius) {};
			\node[font=\small] at ({360/\n * 10}:{\radius + \labeldist}) {10};
			
			\node[draw, blue, circle, fill=blue, inner sep=1.5pt] (11) at ({360/\n * 11}: \newradius) {};
			\node[font=\small, text=blue] at ({360/\n * 11}:{\newradius + \labeldist}) {11};
			
			\node[draw, circle, fill=black, inner sep=1.5pt] (12) at ({360/\n * 12}: \radius) {};
			\node[font=\small] at ({360/\n * 12}:{\radius + \labeldist}) {12};
			
			\node[draw, blue, circle, fill=blue, inner sep=1.5pt] (13) at ({360/\n * 13}: \newradius) {};
			\node[font=\small, text=blue] at ({360/\n * 13}:{\newradius + \labeldist}) {13};
			
			\node[draw, circle, fill=black, inner sep=1.5pt] (14) at ({360/\n * 14}: \radius) {};
			\node[font=\small] at ({360/\n * 14}:{\radius + \labeldist}) {14};
			
			\node[draw, blue, circle, fill=blue, inner sep=1.5pt] (15) at ({360/\n * 15}: \newradius) {};
			\node[font=\small, text=blue] at ({360/\n * 15}:{\newradius + \labeldist}) {15};
			
			\foreach \i/\j in {0/1, 1/2, 2/3, 3/4, 4/5, 5/6, 6/7, 7/8,8/9,9/10,10/11,11/12,12/13,13/14,14/15,15/0} {
				\draw[blue] (\i) -- (\j);
			}
			
			\foreach \i/\j in {0/2, 2/4, 4/6, 6/8, 8/10, 10/12, 12/14, 14/0, 0/4, 4/8, 8/12, 12/0, 0/8} {
				\draw (\i) -- (\j);
			}
		\end{tikzpicture}
		\hspace{-0.1cm}
		\begin{tikzpicture}[scale=0.65]
			\def \n {16} 
			\def \h {15} 
			\def \radius {1.5cm} 
			\def \newradius {1.75cm} 
			\def\labeldist{0.5cm}  
			
			\node[draw, circle, fill=black, inner sep=1.5pt] (0) at ({360/\n * 0}: \radius) {};
			\node[font=\small] at ({360/\n * 0}:{\radius + \labeldist}) {1};
			
			\node[draw, blue, circle, fill=blue, inner sep=1.5pt] (1) at ({360/\n * 1}: \newradius) {};
			\node[font=\small, text=blue] at ({360/\n * 1}:{\newradius + \labeldist}) {10};
			
			\node[draw, circle, fill=black, inner sep=1.5pt] (2) at ({360/\n * 2}: \radius) {};
			\node[font=\small] at ({360/\n * 2}:{\radius + \labeldist}) {3};
			
			\node[draw, blue, circle, fill=blue, inner sep=1.5pt] (3) at ({360/\n * 3}: \newradius) {};
			\node[font=\small, text=blue] at ({360/\n * 3}:{\newradius + \labeldist}) {12};
			
			\node[draw, circle, fill=black, inner sep=1.5pt] (4) at ({360/\n * 4}: \radius) {};
			\node[font=\small] at ({360/\n * 4}:{\radius + \labeldist}) {5};
			
			\node[draw, blue, circle, fill=blue, inner sep=1.5pt] (5) at ({360/\n * 5}: \newradius) {};
			\node[font=\small, text=blue] at ({360/\n * 5}:{\newradius + \labeldist}) {14};
			
			\node[draw, circle, fill=black, inner sep=1.5pt] (6) at ({360/\n * 6}: \radius) {};
			\node[font=\small] at ({360/\n * 6}:{\radius + \labeldist}) {7};
			
			\node[draw, blue, circle, fill=blue, inner sep=1.5pt] (7) at ({360/\n * 7}: \newradius) {};
			\node[font=\small, text=blue] at ({360/\n * 7}:{\newradius + \labeldist}) {0};
			
			\node[draw, circle, fill=black, inner sep=1.5pt] (8) at ({360/\n * 8}: \radius) {};
			\node[font=\small] at ({360/\n * 8}:{\radius + \labeldist}) {9};
			
			\node[draw, blue, circle, fill=blue, inner sep=1.5pt] (9) at ({360/\n * 9}: \newradius) {};
			\node[font=\small, text=blue] at ({360/\n * 9}:{\newradius + \labeldist}) {2};
			
			\node[draw, circle, fill=black, inner sep=1.5pt] (10) at ({360/\n * 10}: \radius) {};
			\node[font=\small] at ({360/\n * 10}:{\radius + \labeldist}) {11};
			
			\node[draw, blue, circle, fill=blue, inner sep=1.5pt] (11) at ({360/\n * 11}: \newradius) {};
			\node[font=\small, text=blue] at ({360/\n * 11}:{\newradius + \labeldist}) {4};
			
			\node[draw, circle, fill=black, inner sep=1.5pt] (12) at ({360/\n * 12}: \radius) {};
			\node[font=\small] at ({360/\n * 12}:{\radius + \labeldist}) {13};
			
			\node[draw, blue, circle, fill=blue, inner sep=1.5pt] (13) at ({360/\n * 13}: \newradius) {};
			\node[font=\small, text=blue] at ({360/\n * 13}:{\newradius + \labeldist}) {6};
			
			\node[draw, circle, fill=black, inner sep=1.5pt] (14) at ({360/\n * 14}: \radius) {};
			\node[font=\small] at ({360/\n * 14}:{\radius + \labeldist}) {15};
			
			\node[draw, blue, circle, fill=blue, inner sep=1.5pt] (15) at ({360/\n * 15}: \newradius) {};
			\node[font=\small, text=blue] at ({360/\n * 15}:{\newradius + \labeldist}) {8};
			
			\foreach \i/\j in {0/1, 1/2, 2/3, 3/4, 4/5, 5/6, 6/7, 7/8,8/9,9/10,10/11,11/12,12/13,13/14,14/15,15/0} {
				\draw[blue] (\i) -- (\j);
			}
			
			\foreach \i/\j in {0/2, 2/4, 4/6, 6/8, 8/10, 10/12, 12/14, 14/0, 0/4, 4/8, 8/12, 12/0, 0/8} {
				\draw (\i) -- (\j);
			}
		\end{tikzpicture}
		\hspace{-0.1cm}
		\begin{tikzpicture}[scale=0.65]
			\def \n {16} 
			\def \h {15} 
			\def \radius {1.5cm} 
			\def \newradius {1.75cm} 
			\def\labeldist{0.5cm}  
			
			\node[draw, circle, fill=black, inner sep=1.5pt] (0) at ({360/\n * 0}: \radius) {};
			\node[font=\small] at ({360/\n * 0}:{\radius + \labeldist}) {2};
			
			\node[draw, blue, circle, fill=blue, inner sep=1.5pt] (1) at ({360/\n * 1}: \newradius) {};
			\node[font=\small, text=blue] at ({360/\n * 1}:{\newradius + \labeldist}) {7};
			
			\node[draw, circle, fill=black, inner sep=1.5pt] (2) at ({360/\n * 2}: \radius) {};
			\node[font=\small] at ({360/\n * 2}:{\radius + \labeldist}) {12};
			
			\node[draw, blue, circle, fill=blue, inner sep=1.5pt] (3) at ({360/\n * 3}: \newradius) {};
			\node[font=\small, text=blue] at ({360/\n * 3}:{\newradius + \labeldist}) {1};
			
			\node[draw, circle, fill=black, inner sep=1.5pt] (4) at ({360/\n * 4}: \radius) {};
			\node[font=\small] at ({360/\n * 4}:{\radius + \labeldist}) {6};
			
			\node[draw, blue, circle, fill=blue, inner sep=1.5pt] (5) at ({360/\n * 5}: \newradius) {};
			\node[font=\small, text=blue] at ({360/\n * 5}:{\newradius + \labeldist}) {11};
			
			\node[draw, circle, fill=black, inner sep=1.5pt] (6) at ({360/\n * 6}: \radius) {};
			\node[font=\small] at ({360/\n * 6}:{\radius + \labeldist}) {0};
			
			\node[draw, blue, circle, fill=blue, inner sep=1.5pt] (7) at ({360/\n * 7}: \newradius) {};
			\node[font=\small, text=blue] at ({360/\n * 7}:{\newradius + \labeldist}) {5};
			
			\node[draw, circle, fill=black, inner sep=1.5pt] (8) at ({360/\n * 8}: \radius) {};
			\node[font=\small] at ({360/\n * 8}:{\radius + \labeldist}) {10};
			
			\node[draw, blue, circle, fill=blue, inner sep=1.5pt] (9) at ({360/\n * 9}: \newradius) {};
			\node[font=\small, text=blue] at ({360/\n * 9}:{\newradius + \labeldist}) {15};
			
			\node[draw, circle, fill=black, inner sep=1.5pt] (10) at ({360/\n * 10}: \radius) {};
			\node[font=\small] at ({360/\n * 10}:{\radius + \labeldist}) {4};
			
			\node[draw, blue, circle, fill=blue, inner sep=1.5pt] (11) at ({360/\n * 11}: \newradius) {};
			\node[font=\small, text=blue] at ({360/\n * 11}:{\newradius + \labeldist}) {9};
			
			\node[draw, circle, fill=black, inner sep=1.5pt] (12) at ({360/\n * 12}: \radius) {};
			\node[font=\small] at ({360/\n * 12}:{\radius + \labeldist}) {14};
			
			\node[draw, blue, circle, fill=blue, inner sep=1.5pt] (13) at ({360/\n * 13}: \newradius) {};
			\node[font=\small, text=blue] at ({360/\n * 13}:{\newradius + \labeldist}) {3};
			
			\node[draw, circle, fill=black, inner sep=1.5pt] (14) at ({360/\n * 14}: \radius) {};
			\node[font=\small] at ({360/\n * 14}:{\radius + \labeldist}) {8};
			
			\node[draw, blue, circle, fill=blue, inner sep=1.5pt] (15) at ({360/\n * 15}: \newradius) {};
			\node[font=\small, text=blue] at ({360/\n * 15}:{\newradius + \labeldist}) {13};
			
			\foreach \i/\j in {0/1, 1/2, 2/3, 3/4, 4/5, 5/6, 6/7, 7/8,8/9,9/10,10/11,11/12,12/13,13/14,14/15,15/0} {
				\draw[blue] (\i) -- (\j);
			}
			
			\foreach \i/\j in {0/2, 2/4, 4/6, 6/8, 8/10, 10/12, 12/14, 14/0, 0/4, 4/8, 8/12, 12/0, 0/8} {
				\draw (\i) -- (\j);
			}
		\end{tikzpicture}
		\hspace{-0.1cm}
		\begin{tikzpicture}[scale=0.65]
			\def \n {16} 
			\def \h {15} 
			\def \radius {1.5cm} 
			\def \newradius {1.75cm} 
			\def\labeldist{0.5cm}  
			
			\node[draw, circle, fill=black, inner sep=1.5pt] (0) at ({360/\n * 0}: \radius) {};
			\node[font=\small] at ({360/\n * 0}:{\radius + \labeldist}) {3};
			
			\node[draw, blue, circle, fill=blue, inner sep=1.5pt] (1) at ({360/\n * 1}: \newradius) {};
			\node[font=\small, text=blue] at ({360/\n * 1}:{\newradius + \labeldist}) {0};
			
			\node[draw, circle, fill=black, inner sep=1.5pt] (2) at ({360/\n * 2}: \radius) {};
			\node[font=\small] at ({360/\n * 2}:{\radius + \labeldist}) {13};
			
			\node[draw, blue, circle, fill=blue, inner sep=1.5pt] (3) at ({360/\n * 3}: \newradius) {};
			\node[font=\small, text=blue] at ({360/\n * 3}:{\newradius + \labeldist}) {10};
			
			\node[draw, circle, fill=black, inner sep=1.5pt] (4) at ({360/\n * 4}: \radius) {};
			\node[font=\small] at ({360/\n * 4}:{\radius + \labeldist}) {7};
			
			\node[draw, blue, circle, fill=blue, inner sep=1.5pt] (5) at ({360/\n * 5}: \newradius) {};
			\node[font=\small, text=blue] at ({360/\n * 5}:{\newradius + \labeldist}) {4};
			
			\node[draw, circle, fill=black, inner sep=1.5pt] (6) at ({360/\n * 6}: \radius) {};
			\node[font=\small] at ({360/\n * 6}:{\radius + \labeldist}) {1};
			
			\node[draw, blue, circle, fill=blue, inner sep=1.5pt] (7) at ({360/\n * 7}: \newradius) {};
			\node[font=\small, text=blue] at ({360/\n * 7}:{\newradius + \labeldist}) {14};
			
			\node[draw, circle, fill=black, inner sep=1.5pt] (8) at ({360/\n * 8}: \radius) {};
			\node[font=\small] at ({360/\n * 8}:{\radius + \labeldist}) {11};
			
			\node[draw, blue, circle, fill=blue, inner sep=1.5pt] (9) at ({360/\n * 9}: \newradius) {};
			\node[font=\small, text=blue] at ({360/\n * 9}:{\newradius + \labeldist}) {8};
			
			\node[draw, circle, fill=black, inner sep=1.5pt] (10) at ({360/\n * 10}: \radius) {};
			\node[font=\small] at ({360/\n * 10}:{\radius + \labeldist}) {5};
			
			\node[draw, blue, circle, fill=blue, inner sep=1.5pt] (11) at ({360/\n * 11}: \newradius) {};
			\node[font=\small, text=blue] at ({360/\n * 11}:{\newradius + \labeldist}) {2};
			
			\node[draw, circle, fill=black, inner sep=1.5pt] (12) at ({360/\n * 12}: \radius) {};
			\node[font=\small] at ({360/\n * 12}:{\radius + \labeldist}) {15};
			
			\node[draw, blue, circle, fill=blue, inner sep=1.5pt] (13) at ({360/\n * 13}: \newradius) {};
			\node[font=\small, text=blue] at ({360/\n * 13}:{\newradius + \labeldist}) {12};
			
			\node[draw, circle, fill=black, inner sep=1.5pt] (14) at ({360/\n * 14}: \radius) {};
			\node[font=\small] at ({360/\n * 14}:{\radius + \labeldist}) {9};
			
			\node[draw, blue, circle, fill=blue, inner sep=1.5pt] (15) at ({360/\n * 15}: \newradius) {};
			\node[font=\small, text=blue] at ({360/\n * 15}:{\newradius + \labeldist}) {6};
			
			\foreach \i/\j in {0/1, 1/2, 2/3, 3/4, 4/5, 5/6, 6/7, 7/8,8/9,9/10,10/11,11/12,12/13,13/14,14/15,15/0} {
				\draw[blue] (\i) -- (\j);
			}
			
			\foreach \i/\j in {0/2, 2/4, 4/6, 6/8, 8/10, 10/12, 12/14, 14/0, 0/4, 4/8, 8/12, 12/0, 0/8} {
				\draw (\i) -- (\j);
			}
		\end{tikzpicture}
		\caption{From left to right the graphs \((V^2,E^2_0)\), \((V^2,E^2_2)\), \((V^2,E^2_1)\), and \((V^2,E^2_3)\) from the proof of~\Cref{lem:optimal-ot-power-2-construct}. The new vertices and edges are indicated in blue.}
		\label{fig:labeling-16}
	\end{figure}
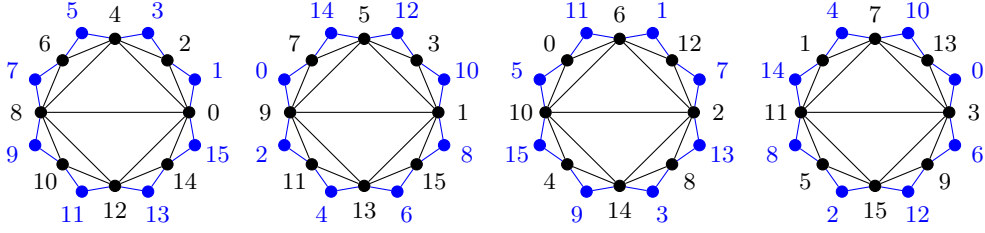
	
	As a next step, we show that the number of \(4t\) vertices for \(t\) edge-disjoint maximal outerplanar graphs is tight.
	\begin{lemma}
		\label{lem:lower-bound}
		Let $t \ge 2$. If there exists a set of \(t\) edge-disjoint maximal outerplanar graphs on \(n\) vertices, then \(n\geq 4t\).
	\end{lemma}
	\begin{proof}
		Since every graph with at most three vertices is outerplanar, $t \ge 2$ implies that $n \ge 4$.
		Let~\((V,E_0), \dots, (V,E_{t-1})\) be edge-disjoint maximal outerplanar graphs on \(n\) vertices.
		Let~\(G\) be the graph with vertex set \(V\) and edge set \(E_0 \cup \dots \cup E_{t-1}\). 
		Since~\(G\) has \(t(2n-3)\) edges, \(t(2n-3) \le \binom{n}{2}\) holds. Thus, we have $n^{2} - (4t+1)n + 6t \ge 0$.
		Let~$r_{1}$ and~$r_{2}$ with~$r_{1} \le r_{2}$ be the roots of the function $f(x) = x^{2} - (4t+1)x + 6t$.
		Since we have~$f(n) \ge 0$, either~$n \le r_{1}$ or~$n \ge r_{2}$ holds.
		A simple calculation shows that $r_{1} < n$ as follows:
		\begin{align*}
			r_{1}
			=
			\frac{(4t+1) - \sqrt{(4t+1)^{2} - 24t}}{2}
			<
			\frac{(4t+1) - (4t-7)}{2}
			=
			4
			\le n,
		\end{align*}
		where $\sqrt{(4t+1)^{2} - 24t} > 4t - 7$ holds by $t \ge 2$.
		Thus we have $n \ge r_{2}$. Since $f(x) \ge 0$ for $x \ge r_{2}$ and $f(4t-1) = -2t+2 < 0$ as $t \ge 2$, we have $4t-1 < r_{2}$.
		This implies that~$n \ge 4t$.
	\end{proof}
	
	Note that for \(t=1\) the statement of~\Cref{lem:lower-bound} does not hold, as $K_3$ is a maximal outerplanar graph with only three vertices.
	
	As a next step, we show that every set of \(t\) edge-disjoint maximal outerplanar graphs on \(n\) vertices can be extended to one on \(n+1\) vertices.
	
	\begin{lemma}
		\label{lem:optimal-increase-vertices}
		Let \(t \in \N\). If there exist \(t\) edge-disjoint maximal outerplanar graphs on \(n\) vertices, then there also exist \(t\) edge-disjoint maximal outerplanar graphs on \(n+1\) vertices.
	\end{lemma}
	
	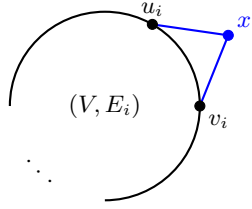
\begin{figure}
		\centering
		\begin{tikzpicture}[scale=1.25]
			\def\r{1} 
			\def\startAngle{180} 
			\def\missingAngle{90} 
			
			\draw[thick]
			({\startAngle+\missingAngle}:\r) arc[start angle={\startAngle+\missingAngle}, end angle={360+\startAngle}, radius=\r];
			
			\draw[decorate, decoration={text along path, text align=center,text={. . .}}]
			({\startAngle}:{\r}) arc[start angle={\startAngle}, end angle={\startAngle+\missingAngle}, radius=\r];
			
			\draw[thick, blue] (0:\r) -- (30:1.5);   
			\draw[thick, blue] (60:\r) -- (30:1.5);  
			
			\filldraw[black] (0:\r) circle (1.5pt) node[below right] {$v_i$};
			\filldraw[black] (60:\r) circle (1.5pt) node[above] {$u_i$};
			\filldraw[blue] (30:1.5) circle (1.5pt) node[above right] {$x$};
			
			\node at (0,0) {\((V,E_i)\)};
		\end{tikzpicture}
		\caption{The outerplanar embedding of \((V\cup \{x\},E_i \cup \{\{u_i,x\},\{v_i,x\}\})\) described in the proof of \Cref{lem:optimal-increase-vertices}. The edge \(\{u_i,v_i\}\) is on the outer cycle of the outerplanar embedding of~\((V,E_i)\). The new vertex \(x\) and the new edges are indicated in blue.}
		\label{fig:triangle-extension}
	\end{figure}	
	
	\begin{proof}
		For $t = 1$, the equivalence of maximality and optimality for outerplanar graphs implies the statement. In the following, we assume that $t \ge 2$. Let~\((V,E_0), \dots, (V,E_{t-1})\) be edge-disjoint maximal outerplanar graphs on \(n\) vertices. By~\Cref{lem:lower-bound}, we have $n \ge 4t$. We show that there is a set of edges~\(\{\{u_0,v_0\}, \dots, \{u_{t-1},v_{t-1}\}\}\) such that each edge \(\{u_i,v_i\}\) is in the outer cycle of~\((V,E_i)\) and the set forms a matching, i.e., no pair of edges in the set shares an endpoint. Assuming that we have such a set of edges, we obtain a set of \(t\) edge-disjoint maximal outerplanar graphs on \(n+1\) vertices as follows: Add a new vertex~\(x\) together with the edges \(\{u_i,x\}\) and \(\{v_i,x\}\) to each \((V,E_i)\). The resulting graphs each have \(2n-3+2=2(n+1)-3\) edges. Further, they are outerplanar as we can take the outerplanar embedding of \((V,E_i)\) where the edge \(\{u_i,v_i\}\) is on the outer cycle and embed~\(x\) next to this edge outside of the outer cycle creating a triangle as depicted in~\Cref{fig:triangle-extension}.
		
		Hence, it remains to show that a set of edges with these properties exists.
		We iteratively construct such a set of edges \(S\), beginning with \(S=\emptyset\). The outer cycle of each~\((V,E_i)\) contains \(n>4(t-1)\) edges. We call an edge on the outer cycle of some~\((V,E_i)\) \emph{blocked}, if at least one of its endpoints is already an endpoint of an edge contained in \(S\). Now we show that for each \({i \in [t]_0}\) we can iteratively add one non-blocked edge of the outer cycle of~\((V,E_i)\) to \(S\):
		In step~\(i\), there are at most \(4i\) blocked edges on the outer cycle of~\((V,E_i)\), as~\(S\) contains \(i\) many edges and each edge corresponds to two vertices, which again each block at most two edges in the outer cycle. Since the outer cycle contains \(n>4(t-1)\) many edges, we find a non-blocked edge in each step. Thus, such a set of edges exists.
	\end{proof}
	
	\section{Optimal outerthickness-$t$ graphs}
	
	In this section, we put together the results to conclude our main theorem.
	
	\begin{theorem}
		\label{thm:main}
		For every $t \in \N_{>0}$ and $n \geq 4t$, there exists an optimal outerthick\-ness-$t$ graph on $n$ vertices.
	\end{theorem}
	
	\begin{proof}
		The existence of an optimal outerthickness \(t\) graph on precisely \(4t\) vertices follows from \Cref{lem:optimal-ot-construct}. Then by \Cref{lem:optimal-increase-vertices}, we obtain optimal outerthickness-\(t\) graphs on every number of \(n \geq 4t\) vertices.
	\end{proof}
	
	Note that we could have alternatively proven a similar statement to \Cref{thm:main} by using~\Cref{lem:optimal-ot-power-2-construct} to obtain optimal outerthickness-\(t\) graphs whenever \(t\) is a power of~$2$, and then forgetting about certain sets of the edge partition. 
	
	Further, the lower bound on the number of vertices from \Cref{lem:lower-bound} allows us to separate the notions of maximal and optimal outerthickness-\(t\) graphs for all values \(t \geq 2\).
	
	\begin{corollary}
		\label{cor:max-not-opt}
		For every \(t \geq 2\), there exist maximal but not optimal outerthick\-ness-\(t\) graphs.
	\end{corollary}
	
	\begin{proof}
		The outerthickness of complete graphs \(K_{4t-4},K_{4t-3},\) and \(K_{4t-2}\) is \(t\)~\cite{Guy1990}. Further, by~\Cref{lem:lower-bound} every optimal outerthick\-ness-\(t\) graph has at least \(4t\) vertices.
	\end{proof}
	
	Besides these easy examples of complete graphs, for \(t=2\) another maximal but not optimal outerthick\-ness-\(2\) graph is given by \(K_7-e\), i.e., the graph obtained from $K_{7}$ by removing an edge.
	A decomposition of \(K_7-e\) into two outerplanar graphs is given in~\Cref{fig:K7-e}.
	The outerthickness of \(K_7\) is $3$~\cite{Guy1990}, and thus $K_{7}-e$ is a maximal outerthick\-ness-\(2\) graph.
	On the other hand, \(K_7-e\) has \(20 < 2 \cdot (2 \cdot 7 - 3)\) edges, and so, it is not an optimal outerthick\-ness-\(2\) graph.
	
	Note that on eight vertices every maximal outerthickness-2 graph is also optimal.
	To see this, first observe that \(K_8-\{e,e'\}\) for two independent edges~\(e,e'\) is an optimal outerthick\-\mbox{ness-2} graph by our construction (cf.~\Cref{fig:labeling-8}).
	On the other hand, for an edge set~$F$ that does not contain two independent edges, \(K_8-F\) has \(K_7\) as a subgraph and hence, it has outerthickness~3.
	
	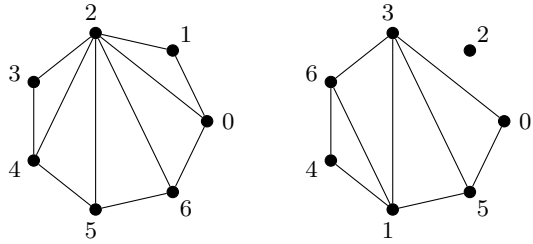
\begin{figure}[t]
		\centering
		\begin{tikzpicture}[scale=0.8]
			\foreach \i in {0,...,6} {
				
				\node[circle, draw, fill=black, inner sep=0pt, minimum size=1.5mm] (N\i) at ({360/7*\i}:1.5cm) {};
				
				\node[font=\small]
				at ({360/7 * \i}:{1.85cm}) {\i};
			}
			
			\draw (N0) -- (N1);
			\draw (N0) -- (N2);
			\draw (N0) -- (N6);
			\draw (N1) -- (N2);
			\draw (N2) -- (N6);
			\draw (N2) -- (N5);
			\draw (N2) -- (N4);
			\draw (N2) -- (N3);
			\draw (N3) -- (N4);
			\draw (N4) -- (N5);
			\draw (N5) -- (N6);
		\end{tikzpicture}
		\hspace{0.5cm}
		\begin{tikzpicture}[scale=0.8]
			\node[circle, draw, fill=black, inner sep=0pt, minimum size=1.5mm] (N0) at ({360/7*0}:1.5cm) {};
			\node[font=\small]
			at ({360/7 * 0}:{1.85cm}) {0};
			
			\node[circle, draw, fill=black, inner sep=0pt, minimum size=1.5mm] (N2) at ({360/7*1}:1.5cm) {};
			\node[font=\small]
			at ({360/7 * 1}:{1.85cm}) {2};
			
			\node[circle, draw, fill=black, inner sep=0pt, minimum size=1.5mm] (N3) at ({360/7*2}:1.5cm) {};
			\node[font=\small]
			at ({360/7 * 2}:{1.85cm}) {3};
			
			\node[circle, draw, fill=black, inner sep=0pt, minimum size=1.5mm] (N6) at ({360/7*3}:1.5cm) {};
			\node[font=\small]
			at ({360/7 * 3}:{1.85cm}) {6};
			
			\node[circle, draw, fill=black, inner sep=0pt, minimum size=1.5mm] (N4) at ({360/7*4}:1.5cm) {};
			\node[font=\small]
			at ({360/7 * 4}:{1.85cm}) {4};
			
			\node[circle, draw, fill=black, inner sep=0pt, minimum size=1.5mm] (N1) at ({360/7*5}:1.5cm) {};
			\node[font=\small]
			at ({360/7 * 5}:{1.85cm}) {1};
			
			\node[circle, draw, fill=black, inner sep=0pt, minimum size=1.5mm] (N5) at ({360/7*6}:1.5cm) {};
			\node[font=\small]
			at ({360/7 * 6}:{1.85cm}) {5};
			
			\draw (N0) -- (N3);
			\draw (N3) -- (N6);
			\draw (N3) -- (N1);
			\draw (N3) -- (N5);
			\draw (N6) -- (N4);
			\draw (N6) -- (N1);
			\draw (N4) -- (N1);
			\draw (N1) -- (N5);
			\draw (N5) -- (N0);
		\end{tikzpicture}
		\caption{A decomposition of \(K_7-e\) for \(e=\{0,4\}\) into two outerplanar graphs. This is an example for a maximal but not optimal outerthickness-2 graph, as well as an outerthickness-2 graph which is not 1-planar.}
		\label{fig:K7-e}
	\end{figure}
	
	We finish this section by concluding, that the existence of the optimal outerthick\-ness-\(2\) graphs allows to partially separate graphs with outerthickness~2 from 1-planar graphs. More precisely we obtain the following.
	
	\begin{corollary}\label{cor:separating-outerthickness2-1planar}
		There exists an infinite family of graphs with outerthickness~2 which are not 1-planar.
	\end{corollary}
	
	\begin{proof}
		A 1-planar graph with \(n\) vertices has at most \(4n-8\) edges~\cite{edge-num-1-planar}. Since an optimal outerthickness-2 graph on \(n\) vertices contains \(4n-6\) edges and exists by \Cref{thm:main}, the statement follows.
	\end{proof}
	
	Besides separating these two generalizations of planar graphs, these graphs are of particular interest since all 1-planar graphs are 6-colorable~\cite{borodin1984solution} and Gethner and Sulanke asked in~\cite{DBLP:journals/gc/GethnerS09} if the lower bound of the chromatic number of outerthickness~2 graphs can be raised above 6. Note that the outerthickness 2 graphs from~\Cref{lem:optimal-ot-construct} as well as~\Cref{lem:optimal-ot-power-2-construct} are both \(K_8\) minus a matching with two edges. Thus, they are both 6-colorable, even though they are not 1-planar.
	
	\section{Conclusion}
	
	We gave two different constructions for \(t\) edge-disjoint maximal outerplanar graphs on~\(4t\) vertices (\Cref{lem:optimal-ot-construct}, \Cref{lem:optimal-ot-power-2-construct}). From this we deduced the existence of optimal outerthick\-ness-\(t\) graphs on every number of \(n \geq 4t\) vertices for all values of \(t \in \N\) (\Cref{thm:main}). The required number of vertices, is tight~(\Cref{lem:lower-bound}).
	
	As a consequence we obtain an infinite family of graphs with outerthick\-\mbox{ness-2} which are not 1-planar (\Cref{cor:separating-outerthickness2-1planar}). The question whether there are 1-planar graphs with outerthickness precisely~3 remains open.
	
	Further, we give examples of maximal but not optimal outerthickness-$t$ graphs for all~$t \geq 2$ (\Cref{cor:max-not-opt}) and raise the question of the existence of an infinite family of such graphs.
	
	Finally, the most pressing open question is to determine the outerthickness of the remaining complete graphs, where all cases but \(n \equiv 3 \bmod 4\) are solved 
	by Guy and Nowakowski~\cite{Guy1990} (see also~\cite{GuyNowakowski1990b}).
	
	\subsection*{Acknowledgements}
	
	Yuto Okada received support from JST SPRING, Grant Number JPMJSP2125. Yota Otachi was partially supported by JSPS KAKENHI Grant Numbers
	JP22H00513, 
	JP24\-H00697, 
	JP25\-K03076, 
	JP25K03077. 
	
\bibliographystyle{plainurl}
\bibliography{ref}
	
\end{document}